\documentclass[12pt]{amsart}
\usepackage{amsfonts, amssymb, latexsym, hyperref, mathrsfs}
\usepackage[usenames, dvipsnames]{xcolor}
\usepackage{ytableau, pbox}
\usepackage{thmtools}
\usepackage{thm-restate}
\usepackage{pgf,tikz}
\usetikzlibrary{calc, shapes, backgrounds,arrows,positioning}
\tikzset{>=stealth',
  head/.style = {fill = white, text=black}, 
  pil/.style={->,thick},
  und/.style={thick},
  junct/.style = {draw,circle,inner sep=0.5pt,outer sep=0pt, fill=black}
  }
\usepackage{enumerate}

\newtheorem{theorem}{Theorem}[section]
\newtheorem{proposition}[theorem]{Proposition}
\newtheorem{lemma}[theorem]{Lemma}

\newtheorem*{claim*}{Claim}
\newtheorem{corollary}[theorem]{Corollary}
\newtheorem{conjecture}[theorem]{Conjecture}
\newtheorem{definition}[theorem]{Definition}

\theoremstyle{remark}
\newtheorem{example}[theorem]{Example}

\newtheorem{remark}[theorem]{Remark}

\numberwithin{equation}{section}

\newcommand{\excise}[1]{}%{$\star$\textsc{#1}$\star$}

\definecolor{qqffqq}{rgb}{0,0.87,0} %green
\definecolor{ttfftt}{rgb}{0,0.87,0} %green
\definecolor{zzqqzz}{rgb}{0.5,0,0.5} %purple
\definecolor{ccwwff}{rgb}{0.5,0,0.5} %purple
\definecolor{cccccc}{rgb}{0.67,0.67,0.67} %grey
\definecolor{uququq}{rgb}{0,0,0} %black
\definecolor{eqeqeq}{rgb}{0.65,0.65,0.65} %grey
\definecolor{uuuuuu}{rgb}{0,0,0} %black
\definecolor{wwwwww}{rgb}{0,0,0} %black
\definecolor{zzttqq}{rgb}{1,1,0} %orange

\begin{document}
\pagestyle{plain}

\mbox{}
\title{Set-valued Skyline Fillings}
\author{Cara Monical}
\address{Dept.~of Mathematics, U.~Illinois at
Urbana-Champaign, Urbana, IL 61801, USA}
\email{cmonica2@illinois.edu}

\date{\today}

\begin{abstract}Set-valued tableaux play an important role in combinatorial $K$-theory.
Separately, semistandard skyline fillings are a combinatorial model for Demazure atoms and key polynomials. 
We unify these two concepts by defining a set-valued extension of semistandard skyline fillings and then give analogues of results of J. Haglund, K. Luoto, S. Mason, and S. van Willigenburg. 
Additionally, we give a bijection between set-valued semistandard Young tableaux and C. Lenart's Schur expansion of the Grothendieck polynomial $G_\lambda$, using the uncrowding operator of V. Reiner, B. Tenner, and A. Yong.
\end{abstract}

\maketitle

\section{Introduction}
Textbook theory of the ring of symmetric functions concerns the Schur basis and its combinatorial model of semistandard Young tableaux.
In enumerative geometry, Schur functions $\{ s_\lambda \}$ are representatives for the Schubert classes in the cohomology ring of the Grassmannian.
The symmetric Grothendieck function $G_\lambda$ is an inhomogeneous deformation of $s_\lambda$ and plays the analogous role in the $K$-theory of the Grassmannian \cite{LasSch1982b}.
A. Buch introduced set-valued tableaux as a combinatorial model for $G_\lambda$, thus providing a $K$-analogue of semistandard Young tableaux \cite{Buch2002}.

In representation theory, Schur functions are the characters of irreducible representations of polynomial $GL_n$ representations.
Similarly, the key polynomials $\{ \kappa_{\lambda,w} \}$ \cite{LasSch1990, ReinerShimozono1995} are the characters of Demazure modules of type $A$ \cite{Demazure1974}. 
Moreover, for fixed $\lambda$, the key polynomials $\{ \kappa_{\lambda,w} \}_{w \in S_n}$ provide an interpolation between the single monomial $x^\lambda$ and the Schur function $s_\lambda$.
A. Lascoux and M.-P. Sch\"utzenberger  introduced Demazure atoms to decompose the key polynomials \cite{LasSch1990}, and thus Demazure atoms give a refinement of $s_\lambda$ into nonsymmetric pieces \cite{HLMvW2011,Mason2008}.
Combinatorially, S. Mason showed Demazure atoms are the generating function for semistandard skyline fillings \cite{Mason2009}.

The main goal of this paper is to unify these two extensions of Schur functions by defining semistandard set-valued skyline fillings.
We then give generalizations of results about ordinary skyline fillings to show how our definition provides a $K$-analogue to Demazure atoms.
This contributes to the study of  $K$-analogues in the realm of algebraic combinatorics, see \cite{LamPyl2007,PechYong16,ReinTenYong16,ThomasYong2009jdt} and the references therein.

\subsection{Background}
A \emph{weak composition} (resp. \emph{composition}) $\gamma$ with $k$ \emph{parts} is a sequence of $k$ nonnegative (resp. positive) integers $\gamma = (\gamma_1,\gamma_2,...,\gamma_k)$, and the \emph{size} of $\gamma$ is $|\gamma| = \sum_i \gamma_i$.  
The \emph{skyline diagram} for $\gamma$ with \emph{basement} $\mathbf{b}=(b_1,...,b_k)$ consists of $k$ left-justified rows with $\gamma_i$ boxes in row $i$, plus an additional column 0 containing the value $b_i$ in row $i$.
Furthermore, a \emph{filling} is an assigment of positive integers to the boxes of the skyline diagram.

Skyline diagrams and fillings were introduced by J. Haglund, M. Haiman, and N. Loehr \cite{HagHaiLoe2008} in their study of the nonsymmetric Macdonald polynomials.  	
Central to the concept of a skyline filling is \emph{triples} which consist of three boxes on two rows $i < j$.  
As pictured, there are two types of triples depending on the relative lengths of the rows.
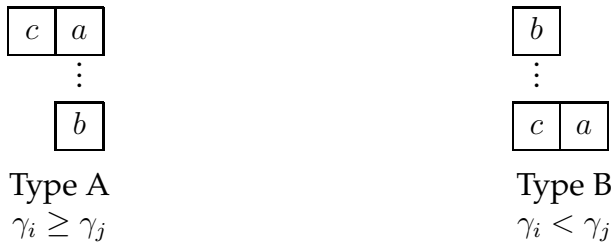
\begin{center}
\begin{minipage}{.4\linewidth}
\begin{center}

\begin{ytableau}
c & a \\
\none & \none[\vdots] \\
\none & b
\end{ytableau} \vspace*{.1in}

Type A \\
$\gamma_i \geq \gamma_j$
\end{center}
\end{minipage}
\begin{minipage}{.4\linewidth}
\begin{center}

\begin{ytableau}
b  \\
\none[\vdots] \\
c & a
\end{ytableau} \vspace*{.1in}

Type B \\
$\gamma_i < \gamma_j$
\end{center}
\end{minipage}
\end{center}
When the rows are weakly decreasing, a triple is an \emph{inversion triple} if $b > c \geq a$ or $c \geq a > b$, and a \emph{coinversion triple} when $a \leq b \leq c$.
A filling is \emph{semistandard} if  
\begin{itemize}
\item [(M1)] entries do not repeat in a column,
\item [(M2)] rows are weakly decreasing (including the basement), and
\item [(M3)] every triple (including those with basement boxes) is an inversion triple.
\end{itemize}
The notion of a \emph{semistandard skyline filling} is due to S. Mason \cite{Mason2009} through her study of the Demazure atom.
Given a filling $F$, the \emph{content} of $F$ is the weak composition $\delta$ where $\delta_i$ is the number of $i$s in $F$, excluding any $i$s in the basement.
Then, the monomial $x^F$ is  $x^\delta = x_1^{\delta_1}x_2^{\delta_2}...x_k^{\delta_k}$, and the \emph{size} of $F$, denoted $|F|$, is $|\delta|$.   
Finally, the Demazure atom $A_\gamma$ is \[ A_\gamma = \sum_F x^F \] where the sum runs over all semistandard skyline fillings of $\gamma$ with basement $b_i  =i$ \cite{Mason2009}.

\subsection{Definition of Set-Valued Skyline Fillings}

We now extend the notion of semistandard fillings to set-valued fillings.
A \emph{set-valued filling} is an assignment of non-empty subsets of positive integers to the boxes of the skyline diagram.
The maximum entry in each box is the \emph{anchor entry} and all other entries are \emph{free entries}.
A set-valued filling is \emph{semistandard} if
\begin{itemize}
\item [(S1)]  entries do not repeat in a column,
\item[(S2)]  rows are weakly decreasing  where sets $A \geq B$ if $\min A \geq \max B$,
\item[(S3)]  every triple of anchor entries is an inversion triple, and
\item[(S4)]  free entries are in the highest (lowest numbered) row that does violate (S2).
\end{itemize}
Note that the concept of anchor entries is a key part of this definition, see Remark \ref{rmk:KeyPoint}.
Examples of semistandard set-valued skyline fillings, with their corresponding monomials, are given below where anchor entries are given in bold.

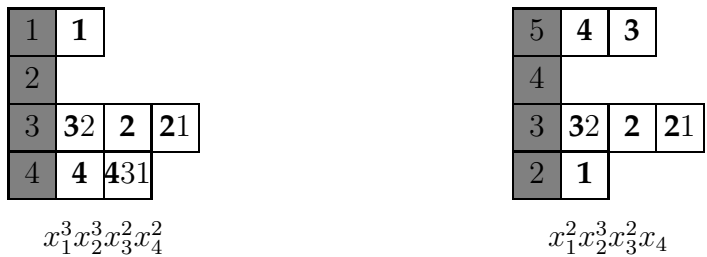
\begin{center}
\begin{minipage}{.4\linewidth}
\begin{center}
	\begin{ytableau}
	*(gray) 1 & \textbf{1}  \\ 
	*(gray) 2 \\
	*(gray) 3 & \textbf{3}2 & \textbf{2} & \textbf{2}1\\
	*(gray) 4 & \textbf{4} & \textbf{4}31
	\end{ytableau} \\\vspace*{.1in}
	$x_1^3x_2^3x_3^2x_4^2$
\end{center}
\end{minipage}
\begin{minipage}{.4\linewidth}
\begin{center}
	\begin{ytableau}
	*(gray) 5 & \textbf{4} & \textbf{3} \\ 
	*(gray) 4 \\
	*(gray) 3 & \textbf{3}2 & \textbf{2} & \textbf{2}1\\
	*(gray) 2 & \textbf{1}
	\end{ytableau} \\\vspace*{.1in}
$x_1^2x_2^3x_3^2x_4$
\end{center}
\end{minipage}
\end{center}
By analogy with the Demazure case, we define combinatorial Lascoux atoms as the generating function for semistandard set-valued skyline fillings.  

\begin{definition} For $\gamma$ let ${\sf SetSkyFill}(\gamma)$ be the set of semistandard set-valued skyline diagrams of shape $\gamma$ and basement $b_i = i$.  Then the combinatorial Lascoux atom $\mathcal{L}_\gamma$ is 
\[ \mathcal{L}_\gamma(x_1,...,x_{|\gamma|}; \beta) = \sum_{F \in {\sf SetSkyFill}(\gamma)} \beta^{|F|-|\gamma|}x^F. \] 
\label{defn:lasAtom}
\end{definition} 
Figure \ref{fig:skylineTable} gives $\mathcal{L}_\gamma$ and the corresponding fillings for weak compositions that are rearrangements of $(2,1,0)$.  
Clearly, setting $\beta = 0$ yields $A_\gamma$, and thus $\mathcal{L}_\gamma$ is a inhomogeneous deformation of $A_\gamma$.  
This mostly shows combinatorial Lascoux atoms form a new (finite) basis of ${\sf Pol} = \mathbb{Z}[x_1,x_2,...]$ -- this is Proposition \ref{prop:lasBasis}.   

Definition \ref{defn:lasAtom} is our $K$-analogue of the Demazure atom.  Conjecturally, it satisfies the natural recurrence for $K$-theoretic Demazure atoms (see Conjecture \ref{conj:LasConj}) and we will give generalizations of earlier results that hold for combinatorial Lascoux atoms to support this view.

\begin{figure}
\renewcommand{\arraystretch}{2.5}
\begin{tabular}{|c||p{6cm}|p{9cm}|}
\hline
$\mathcal{L}_{210}$ & $x_1^2x_2$ & \small\begin{ytableau} *(gray) 1 & 1 & 1 \\ *(gray) 2 & 2 \\ *(gray) 3 \end{ytableau} \\[.4in] \hline 
$\mathcal{L}_{201}$  & $x_1^2x_2$ & \small\begin{ytableau} *(gray) 1 & 1 & 1 \\ *(gray) 2  \\ *(gray) 3 & 3 \end{ytableau} \, \begin{ytableau} *(gray) 1 & 1 & 1 \\ *(gray) 2  \\ *(gray) 3 & 32 \end{ytableau} \\[.4in] \hline
$\mathcal{L}_{120}$  & $x_1x_2^2 + \beta x_1^2x_2^2$ & \small \begin{ytableau} *(gray) 1 & 1  \\ *(gray) 2 & 2 & 2 \\ *(gray) 3 \end{ytableau} \,  \begin{ytableau} *(gray) 1 & 1  \\ *(gray) 2 & 21 & 2 \\ *(gray) 3 \end{ytableau} \\[.4in] \hline
$\mathcal{L}_{021}$  & \pbox{6cm}{\vspace*{.35in}$(x_1x_2x_3 + x_2^2x_3) +$ \\ $\beta(2x_1x_2^2x_3 + x_1^2x_2x_3) +$ \\ $\beta^2x_1^2x_2^2x_3$}  & \small\begin{ytableau}
*(gray) 1 \\
*(gray) 2 & 2 & 1 \\
*(gray) 3 & 3 \\
\end{ytableau}
\begin{ytableau}
*(gray) 1 \\
*(gray) 2 & 2 & 2 \\
*(gray) 3 & 3 \\
\end{ytableau}
\begin{ytableau}
*(gray) 1 \\
*(gray) 2 & 2 & 2 \\
*(gray) 3 & 31 \\
\end{ytableau}
\begin{ytableau}
*(gray) 1 \\
*(gray) 2 & 2 & 21 \\
*(gray) 3 & 3 \\
\end{ytableau}
\begin{ytableau}
*(gray) 1 \\
*(gray) 2 & 2 & 21 \\
*(gray) 3 & 31 \\
\end{ytableau} \hspace*{.43cm}
\begin{ytableau}
*(gray) 1 \\
*(gray) 2 & 21 & 1 \\
*(gray) 3 & 3 \\
\end{ytableau} \\[.45in] \hline
$\mathcal{L}_{102}$  & \pbox{6cm}{\vspace*{.35in}$(x_1x_2x_3 + x_1x_3^2) + \beta(x_1^2x_2x_3 + $\\$x_1^2x_3^2 + x_1x_2x_3^2 + x_1x_2^2x_3) +$\\$ \beta^2(x_1^2x_2x_3^3 + x_1^2x_2^2x_3)$} &\small \begin{ytableau}
*(gray) 1 & 1 \\
*(gray) 2 \\
*(gray) 3 & 3 & 2 \\
\end{ytableau} \hspace*{.1in}
\begin{ytableau}
*(gray) 1 & 1 \\
*(gray) 2 \\
*(gray) 3 & 3 & 3 \\
\end{ytableau}\hspace*{.1in}
\begin{ytableau}
*(gray) 1 & 1 \\
*(gray) 2 \\
*(gray) 3 & 3 & 21 \\
\end{ytableau}\hspace*{.1in}
\begin{ytableau}
*(gray) 1 & 1 \\
*(gray) 2 \\
*(gray) 3 & 3 & 31 \\
\end{ytableau}
\begin{ytableau}
*(gray) 1 & 1 \\
*(gray) 2 \\
*(gray) 3 & 3 & 32 \\
\end{ytableau} \hspace*{.63cm}
\begin{ytableau}
*(gray) 1 & 1 \\
*(gray) 2 \\
*(gray) 3 & 3 & 321 \\
\end{ytableau}\hspace*{.45cm}
\begin{ytableau}
*(gray) 1 & 1 \\
*(gray) 2 \\
*(gray) 3 & 32 & 2 \\
\end{ytableau}\hspace*{.45cm}
\begin{ytableau}
*(gray) 1 & 1 \\
*(gray) 2 \\
*(gray) 3 & 32 & 21 \\
\end{ytableau}  \\[.58in] \hline
$\mathcal{L}_{012}$  & $x_2x_3^2 + \beta(2x_1x_2x_3^2 + x_2^2x_3^2) + \beta^2 (x_1x_2^2x_3^2 + x_1^2x_2x_3^2 + x_1x_2^2x_3^2) + \beta^3 x_1^2x_2^2x_3^2$ & \small \begin{ytableau}
*(gray) 1 \\
*(gray) 2 & 2 \\
*(gray) 3 & 3 & 3 \\
\end{ytableau}\hspace*{.1cm}
\begin{ytableau}
*(gray) 1 \\
*(gray) 2 & 2 \\
*(gray) 3 & 3 & 31 \\
\end{ytableau}\hspace*{.1cm}
\begin{ytableau}
*(gray) 1 \\
*(gray) 2 & 2 \\
*(gray) 3 & 3 & 32 \\
\end{ytableau}\hspace*{.1cm}
\begin{ytableau}
*(gray) 1 \\
*(gray) 2 & 2 \\
*(gray) 3 & 3 & 321 \\
\end{ytableau}
\begin{ytableau}
*(gray) 1 \\
*(gray) 2 & 21 \\
*(gray) 3 & 3 & 3 \\
\end{ytableau}\hspace*{.5cm}
\begin{ytableau}
*(gray) 1 \\
*(gray) 2 & 21 \\
*(gray) 3 & 3 & 31 \\
\end{ytableau}\hspace*{.53cm}
\begin{ytableau}
*(gray) 1 \\
*(gray) 2 & 21 \\
*(gray) 3 & 3 & 32 \\
\end{ytableau} \hspace*{.45cm}
\begin{ytableau}
*(gray) 1 \\
*(gray) 2 & 21 \\
*(gray) 3 & 3 & 321 \\
\end{ytableau}\vspace*{.1in}\\
\hline
\end{tabular}
\label{fig:skylineTable}
\caption{This example gives $\mathcal{L}_\gamma$ and the corresponding semistandard set-valued skyline fillings for weak compositions that are rearrangements of $(2,1,0)$.}
\end{figure} 

\subsection{Main Results}
A \emph{partition} is a weak composition such that the parts are weakly decreasing, and for $\gamma$, define $\lambda(\gamma)$ as the unique partition with the same parts as $\gamma$.
As $s_\lambda = \sum_{\lambda(\gamma) = \lambda} A_\gamma$, the Demazure atoms are a nonsymmetric refinement of the Schur functions \cite{Mason2009}.
We generalize this to $G_\lambda$ and $\mathcal{L}_\gamma$, the $K$-analogues of $s_\lambda$ and $A_\gamma$, respectively. 
\begin{restatable}{theorem}{GLasDecomp}
\[ G_\lambda = \sum_{\lambda(\gamma) = \lambda} \mathcal{L}_\gamma .\] 
\label{thm:GLasDecomp}
\end{restatable}

In between the ring of formal power series and the ring of symmetric functions, is $QSym$, the ring of quasisymmetric functions.  
A function $f$ is \emph{quasisymmetric} if for any positive integers $\alpha_1, ..., \alpha_k$ and strictly increasing sequence of positive integers $i_1 < ... < i_k$,
\[ [x_{i_1}^{\alpha_1}... x_{i_k}^{\alpha_k}] f = [x_1^{\alpha_1}...x_k^{\alpha_k}] f. \]

J.~Haglund, K. Luoto, S. Mason, and S. van Willigenburg also use Demazure atoms to define the \emph{quasisymmetric Schur functions} $\{ \mathcal{S}_\alpha\}$, which provide a quasisymmetric refinement of the Schur functions \cite{HLMvW2011QS}.
We generalize \cite[Definition 5.1]{HLMvW2011QS} to define the quasisymmetric Grothendieck functions.
\begin{definition}
For a composition $\alpha$, the quasisymmetric Grothendieck function $\mathcal{G}_\alpha$ is  
\[ \mathcal{G}_\alpha = \sum_{\gamma^+ = \alpha} \mathcal{L}_\gamma \] where $\gamma^+$ is the composition formed by dropping zero parts from $\alpha$.
\label{defn:QuasiGroth}
\end{definition}
By combining Theorem \ref{thm:GLasDecomp} and Definition \ref{defn:QuasiGroth}, we decompose $G_\lambda$ into quasisymmetric Grothendieck functions which generalizes the decomposition in \cite[pg. 13]{HLMvW2011QS}.
\begin{corollary}
\[ G_\lambda = \sum_{\lambda(\alpha) = \lambda} \mathcal{G}_\alpha. \]
\end{corollary}

Furthermore, $\{ \mathcal{G}_\alpha \}$ does in fact form a basis of $QSym$, generalizing \cite[Proposition 5.5]{HLMvW2011QS} and thus using our $K$-analogue of $\mathcal{A}_\gamma$, we are able to provide a $K$-analogue of $\mathcal{S}_\alpha$.

\begin{restatable}{theorem}{QuasiGroth}
As $\alpha$ runs over all compositions, the functions $\{\mathcal{G}_\alpha\}$ form a basis for $QSym$.
\label{thm:QuasiGrothBasis}
\end{restatable}

As seen below, the expansion of a power series $f$ into Lascoux atoms allows us to determine if $f$ is quasisymmetric or symmetric.
If it is, the expansion allows us to determine if $f$ is $\mathcal{G}_\alpha$- or $G_\lambda$-positive, which is often of interest, cf. \cite[Section 1.1]{LMvW2013}.
\begin{restatable}{proposition}{Refinement} 
 Suppose $f = \sum_\gamma c_\gamma \mathcal{L}_\gamma$.  Then 
\begin{enumerate}
\item $f$ is quasisymmetric if and only if $c_\gamma = c_\delta$ for all $\gamma^+ = \delta^+$, and
\item $f$ is symmetric if and only if $c_\gamma = c_\delta$ for all $\lambda(\gamma) = \lambda(\delta)$.
\end{enumerate}
Furthermore, if $f$ is quasisymmetric (resp. symmetric), $f$ is $\mathcal{G}_\alpha$-positive (resp. $G_\lambda$-positive) if and only if $f$ is $\mathcal{L}_\gamma$-positive.
\label{prop:Refinement}
\end{restatable}

The next two sections further investigate Lascoux atoms and quasisymmetric Grothen-dieck functions.  
Then in section 4, we provide a link between ordinary and set-valued tableaux through a bijection between Lenart's Schur expansion of $G_\lambda$ and set-valued tableaux that produces a pair of tableaux from a set-valued tableaux using the uncrowding operation of V. Reiner, B. Tenner, and A. Yong \cite{ReinTenYong16}.
Finally, in section 5, we state further conjectures about Lascoux atoms that continue the analogy with Demazure atoms.
\section{Combinatorial Lascoux Atoms}
We first show that combinatorial Lascoux atoms form a finite basis of ${\sf Pol} = \mathbb{Z}[x_1,.x_2,...]$.
Let $\prec$ be the lexicographic order on monomials.

\begin{lemma} For $k = \max \gamma$,
\[ \mathcal{L}_\gamma = x^\gamma + \sum_{\substack{\delta \prec \gamma \\ \max \delta \leq k}} c_{\gamma,\delta} \beta^{|\delta|-|\gamma|}x^\delta. \]
\label{lem:atomLexLargest}
\end{lemma}
\begin{proof}
Since we are considering skyline fillings with basement $b_i = i$ and rows are weakly decreasing (S2), for any $i_0$, the boxes in the first $i_0$ rows can only have the values $1,...,i_0$.

Then, we first consider skyline fillings of shape $\gamma$ and content $\gamma$.  
There can be no free entries, as we have exactly as many entries as we have boxes.
Since the first row can only contain 1s and we have $\gamma_1$ boxes in the first row and $\gamma_1$ entries with value 1, all 1s must be placed in the first row.
Likewise, the second row can only contain 1s and 2s.  
However all 1s were placed in row 1, and so we have $\gamma_2$ boxes in row 2 and exactly $\gamma_2$ 2s that can be placed in the second row.  
Thus all the 2s must be placed in the second row.
Proceeding in this manner, we see row $i$ must contain all $i$s. 
Thus, $x^\gamma$ appears in $\mathcal{L}_\gamma$ with coefficient 1 because the filling formed by filling row $i$ with all $i$s for anchor entries and no free entries is the unique element of ${\sf SetSkyFill}(\gamma)$ with content $\gamma$.

Now suppose $\delta \succ \gamma$ and we will show there is no element of ${\sf SetSkyFill}(\gamma)$ with content $\delta$.
Since $\delta \succ \gamma$, there exists $i_0$ such that \[ \sum_{i=1}^{i_0} \gamma_i > \sum_{i=1}^{i_0} \delta_i. \]
However, $\displaystyle \sum_{i=1}^{i_0} \gamma_i$ is the number of boxes in the first $i_0$ rows and $\displaystyle \sum_{i=1}^{i_0} \delta_i$ is the number of instances of the numbers $1, ..., i_0$.
Thus there are more boxes in rows $1,...,i_0$ than instances of the numbers $1,...,i_0$, and so at least one box in rows $1,...,i_0$ must be empty.  
Then no element of ${\sf SetSkyFill}(\gamma)$ has content $\delta$. 

Finally consider $F \in {\sf SetSkyFill}(\gamma)$ of content $\delta$.
Since $F$ has $k$ columns, excluding the basement, and numbers cannot repeat in a column (S1), each $i$ can appear at most $k$ times in $F$. 
Thus $\max{\delta} \leq k$.
\end{proof}

\begin{proposition}
As $\gamma$ ranges over all weak compositions, $\{ \mathcal{L}_\gamma \}$ forms a finite basis for ${\sf Pol}$.
\label{prop:lasBasis}
\end{proposition}
\begin{proof}
First, we consider the expansion of $x^\gamma$ into Lascoux atoms with $k = \max \gamma$.
By Lemma \ref{lem:atomLexLargest}, \[ x^\gamma = \mathcal{L}_\gamma -\sum_{\substack{\delta \prec \gamma \\ \max \delta \leq k}} c_{\gamma,\delta} \beta^{|\delta|-|\gamma|}x^\delta. \]
Let $\delta_1$ be the lexicographically largest term such that $c_{\gamma,\delta_1} \neq 0$, and observe since $\delta_1 \prec \gamma$ and $\max \delta_1 \leq \max \gamma = k$, for any term with $c_{\delta_1,\delta} \neq 0$,  $\max \delta \leq k$.
Then,
\[
x^\gamma = \mathcal{L}_\gamma - c_{\gamma,\delta_1} \beta^{|\delta_1|-|\gamma|} \mathcal{L}_{\delta_1}+ \sum_{\substack{\delta \prec \delta_1 \\ \max \delta \leq k}} (c_{\gamma,\delta_1}c_{\delta_1,\delta}-c_{\gamma,\delta}) \beta^{|\delta|-|\gamma|}x^\delta. \]
We then iterate this process with the lexicographically largest term remaining in the sum, and thus after the $i$th step,
\[ x^\gamma = \mathcal{L}_\gamma + \sum_{j = 1}^i a_{\gamma,\delta_j} \beta^{|\delta_j|-|\gamma|}\mathcal{L}_{\delta_j} + \sum_{\substack{\delta \prec \delta_i \\ \max \delta \leq k}} b_{\gamma,\delta} \beta^{|\delta|-|\gamma|}x^\delta. \]  
Since we take the lexicographically largest term remaining at each step, for all $i$, $\max \delta_i \leq k$ and $\gamma \succ \delta_1 \succ \delta_2 \succ ... \succ \delta_i$
Since there are finitely many compositions lexicographically smaller than $\gamma$ with maximum part at most $k$, this process must terminate and we have a finite expansion of $x^\gamma$ into Lascoux atoms.

Since any monomial $x^\gamma$ has a finite expansion in Lascoux atoms, any $f \in {\sf Pol}$ does as well.
Finally, suppose \[ 0 = \sum_\gamma c_\gamma \mathcal{L}_\gamma \] and by setting $\beta = 0$,
\[ 0 = \sum_\gamma c_\gamma A_\gamma. \]  
Since Demazure atoms form a basis of polynomials, $c_\gamma = 0$ for all $\gamma$, and $\{ \mathcal{L}_\gamma \}$ is a linearly independent set.
\end{proof}

We now define the bijections $\hat{\rho}$ and $\hat{\rho}^{-1}$ used to prove the $G_\lambda$ expansion of Theorem \ref{thm:GLasDecomp}.  
In the special case where there are no free entries, these are precisely the bijections $\rho$ and $\rho^{-1}$ given by Mason in \cite{Mason2008}.

A \emph{set-valued reverse tableaux} is a filling of the shape $\lambda$ with non-empty sets of positive integers with weakly (resp. strictly)  rows (resp. columns).
We use the convention that $G_\lambda$ is the sum over set-valued reverse tableaux. 
Recall ${\sf SetSkyFill}(\gamma)$ is the collection of set-valued skyline fillings of shape $\gamma$ and basement $b_i = i$ and let ${\sf SetRT}(\lambda)$ be the collection of set-valued reverse tableaux of shape $\lambda$.
Then, we define the map \[ \hat{\rho}: \bigcup_{\lambda(\gamma) = \lambda} {\sf SetSkyFill}(\gamma) \rightarrow {\sf SetRT}(\lambda)\] as follows.
First, sort the anchor entries of each column into decreasing order and then place the free entries in the unique box in their column such that the columns remain strictly decreasing and the free entries remain free.

For the inverse $\hat{\rho}^{-1}$, start with an empty skyline diagram with basement $b_i = i$.
Work by columns left to right, top to bottom and place each anchor entry in the first row such that weakly decreasing rows is preserved.  
When all anchor entries have been placed, place the free entries in the highest box in their column such that the rows are weakly decreasing and the free entries remain free.

\begin{example} Given the filling $F = $
\begin{ytableau}
	*(gray) 1 & \textbf{1}  \\ 
	*(gray) 2 \\
	*(gray) 3 & \textbf{3}2 & \textbf{2} & \textbf{2}1\\
	*(gray) 4 & \textbf{4} & \textbf{4}31\\ 
  *(gray) 5 & \textbf{5} \\ 
	\end{ytableau}
	we calculate $\hat{\rho}(F) = $
		\begin{ytableau}
\textbf{5}  & \textbf{4}3  & \textbf{2}1\\
\textbf{4} & \textbf{2}1 \\
 \textbf{3}2  \\
\textbf{1} \\
\end{ytableau}.
\end{example}

\begin{theorem}
\label{thm:rhoBijection}
The map $\hat{\rho}$ is a bijection and $\hat{\rho}^{-1}$ is its inverse.
\end{theorem}
\begin{proof}
The anchor entries of a semistandard set-valued skyline filling form an ordinary semistandard skyline filling, and likewise the anchor entries of set-valued reverse tableau form an ordinary reverse tableau.
Since $\hat{\rho}$ and $\hat{\rho}^{-1}$ act exactly on the anchor entries by $\rho$ and $\rho^{-1}$ which are well-defined and mutual inverses \cite{Mason2008}, $\hat{\rho}$ and $\hat{\rho}^{-1}$ are well-defined and mutual inverses on the anchor entries.
Thus, since the anchor entries determine the shape of the resulting filling or tableau, $\hat{\rho}$ and $\hat{\rho}^{-1}$ produce fillings and tableaux of the correct shapes.
Thus we only need to show $\hat{\rho}$ and $\hat{\rho}^{-1}$ are well-defined and mutual inverses on the free entries.

{\sf The map $\hat{\rho}$ is well-defined on the free entries.}
 Let $F \in {\sf SetSkyFill}(\gamma)$ and we want to show that $T = \hat{\rho}(F) \in {\sf SetRT}(\lambda(\gamma))$.  
By construction the columns of $T$ are strictly decreasing and so we only need to show that the rows of $T$ are weakly decreasing.  
Suppose to the contrary that row $i$ is not and then there must be a free entry $\alpha \in T(i,j)$ for some $j$ such that \[ \alpha < \max(T(i,j+1)). \]  
Since $\alpha$ is a free entry in the box $(i,j)$, \[\max(T(1,j)) > \max(T(2,j)) > ... > \max(T(i,j)) > \alpha > \max(T(i+1,j)) > ... > \max(T(k,j)). \]  
Since the anchor entries of column $j$ of $T$ are precisely the anchor entries of column $j$ of $F$, there are exactly $i$ anchor entries bigger than $\alpha$ in column $j$ of $F$.  
However, since \[ \alpha < \max(T(i,j+1)) < \max(T(i-1,j+1)) < ... < \max(T(1,j+1), \] there are at least $i$ anchor entries bigger than $\alpha$ in column $j+1$ of $F$.

Since $\alpha$ is a free entry of $F$, it must be in the box of one of the $i$ anchor entries $\max(T(k,j))$ for $k = 1,...,i$.  
However, since the rows of $F$ are weakly decreasing (S2), none of the $i$ anchor entries $\max(T(k,j+1)$ for $k = 1,...,i$ can appear to the right of the box of $\alpha$.
Thus there are at most $i-1$ anchor entries in column $j$ that can appear to the left of $i$ anchor entries in column $j+1$ that are bigger than $\alpha$, contradicting that $F \in {\sf SetSkyFill}(\gamma)$. 

{\sf The map $\hat{\rho}^{-1}$ is well-defined on the free entries.}
Let $T \in {\sf SetRT}(\lambda)$ and we want to show $\hat{\rho}^{-1}(T) = F \in {\sf SetSkyFill}(\gamma)$.  
Since by definition $\hat{\rho}^{-1}$ places free entries in the highest possible row that does not break weakly decreasing, we only need to show that such a row exists.  
Thus suppose $\alpha$ is a free entry of $T$ in box $(i,j)$.  
Since \[ \max(T(1,j)) > \max(T(2,j)) > ... > \max(T(i,j)) > \alpha \] there are $i$ anchor entries in column $j$ of $F$ that are bigger than $\alpha$, and so $\alpha$ can be placed in any of these $i$ boxes and remain free.

Since the rows of $T$ are weakly decreasing, \[ \alpha \geq \max(T(i,j+1)) > \max(T(i+1,j+1) > ... .\]  
Then there are at most $i-1$ anchor entries of column $j+1$ that are bigger than $\alpha$.  
Since these at most $i-1$ entries in column $j+1$ cannot appear to the right of all $i$ possibilities in column $j$, $\alpha$ can be placed in one of the $i$ boxes in column $j$ where the anchor entry is larger than $\alpha$.

{\sf The maps $\hat{\rho}$ and $\hat{\rho}^{-1}$ are mutual inverses.}
Since the columns of  a reverse tableaux are strictly decreasing, there is at most one set-valued reverse tableaux of fixed anchor and free entries in each column.
Thus $\hat{\rho}\hat{\rho}^{-1}(T) = T$ as both $\hat{\rho}$ and $\hat{\rho}^{-1}$ preserve the anchor and free entries of each column of $T$.

For the same reason, $\hat{\rho}^{-1}\hat{\rho}(F) = F$.  In \cite{Mason2008}, Mason showed there is at most one semistandard skyline filling (of any shape) with basement $b_i = i$ with given entries in each column.  Since the anchor entries form a semistandard skyline filling and free entries are required to be in the highest possible row, there is at most one set-valued semistandard skyline filling (of any shape) with basement $b_i = i$ with given anchor and free entries in each column.  
\end{proof}

This proves the decomposition of $G_\lambda$ into Lascoux atoms.
\GLasDecomp*

\begin{remark}
\label{rmk:KeyPoint}
One might expect a semistandard set-valued skyline filling to be a filling such that any selection of one number from each box is a semistandard skyline filling.  
However, then the left tableau below would not be semistandard as the right tableau 
violates the (M3) condition in rows 2 and 3.  Compare this with \cite[Section 1.2]{MillKnutYong2008}.
\begin{center}
	\begin{ytableau} 
*(gray) 1 & 1 & 1 \\
*(gray) 2 \\
*(gray) 3 & 32
\end{ytableau}
\hspace*{1in}
	\begin{ytableau}
	*(gray) 1 & 1 & 1 \\ 
	*(gray) 2 \\
	*(gray) 3 & 2 \\
	\end{ytableau}
	\end{center}
\end{remark}

\section{Quasisymmetric Grothendieck Functions}

Recall that a function $f$ is quasisymmetric if for any positive integers $\alpha_1, ..., \alpha_k$ and strictly increasing sequence of positive integers $i_1 < i_2 < ... < i_k$,
$[x_{i_1}^{\alpha_1}... x_{i_k}^{\alpha_k}] f = [x_1^{\alpha_1}...x_k^{\alpha_k}] f$.
Equivalently $f$ is quasisymmetric if and only if for all $i$, $f$ is invariant under switching $x_i$ and $x_{i+1}$ \emph{except not in monomials that have both}.
Thus when $f$ is modeled by a combinatorial object, we expect $f$ to be quasisymmetric if and only if the combinatorial object is governed by rules depending \emph{only} on relative order.

Therefore Demazure and Lascoux atoms are not quasisymmetric because the basement $b_i = i$ forces the anchor entry at position $(i,1)$, if it exists, to be $i$.
In \cite{HLMvW2011QS}, the quasisymmetric Schur function was originally defined
 $ \mathcal{S}_\alpha = \sum_{\gamma^+ = \alpha} A_\gamma$, where $\gamma^+$ is the composition formed from $\gamma$ by dropping all zeros.
 $\mathcal{S}_\alpha$ was then shown to be the sum over  \emph{semistandard composition tableau}, skyline fillings of a composition $\alpha$ with no basement and strictly increasing entries from top to bottom along the first column.
 
 Thus we define \emph{semistandard set-valued composition tableau} as fillings of a composition shape $\alpha$ with non-empty subsets of positive integers such that
 \begin{itemize}
 \item[(Q1)] entries weakly decrease along rows,
 \item[(Q2)] anchor entries form a semistandard composition tableaux, and
 \item[(Q3)] free entries are in the highest row such that (Q1) is not violated.
 \end{itemize}
 
 Let ${\sf SetCompTab}(\alpha)$ be the collection of semistandard set-valued composition tableaux of shape $\alpha$.
 Adding zero rows to $\alpha$ allows the anchor entries of the first column to be any increasing sequence, and thus \begin{equation} \mathcal{G}_\alpha := \sum_{\gamma^+ = \alpha} \mathcal{L}_\gamma = \sum_{T \in {\sf SetCompTab}(\alpha)} \beta^{|T|-|\alpha|}x^T. \end{equation}
 
Since all the rules governing a semistandard set-valued composition tableau only depend on the relative order of the entries in each box, we expect $\mathcal{G}_\alpha$ to be quasisymmetric.

\begin{proposition}
The function $\mathcal{G}_\alpha$ is quasisymmetric.
\end{proposition}
\begin{proof}
Fix a composition $\alpha$ and fix $i$.  We need to show that $\mathcal{G}_\alpha$ is invariant under switching $x_i$ and $x_{i+1}$ but not in monomials that contain both.
Thus, it suffices to show \[ \#\{F \in {\sf SetCompTab}(\alpha) : F \text{ has content } \beta = (\beta_1,...,\beta_{i-1},\beta_i,0,\beta_{i+2},...,\beta_n) \} \] equals
\[ \#\{F \in {\sf SetCompTab}(\alpha) : F \text{ has content } \hat{\beta} =  (\beta_1,...,\beta_{i-1},0,\beta_i,\beta_{i+2},...,\beta_n) \}. \]  

Suppose $T \in {\sf SetCompTab}(\alpha)$ has no $i+1$s and let $\hat{T}$ be the tableau formed from $T$ by replacing all $i$s replaced with $i+1$s. 
We will show that $\hat{T}$ is semistandard.

Let $\max(T(r,c))$ and ${\tt free}(T(r,c))$ be the anchor entry and set of free entries, respectively, of the box $(r,c)$ in $T$.
First, if $\max(T(r,c)) = i$, then $\max(\hat{T}(r,c)) = i+1$, and thus anchor $i$s become anchor $i+1$s.
Furthermore, if $i \in {\tt free}(T(r,c))$, then $\max(T(r,c)) \geq i+2$ because $T$ has no $i+1$s.  
Thus $i+1 < \max(\hat{T}(r,c)) = \max(T(r,c))$ and so $i+1 \in {\tt free}(\hat{T}(r,c))$.
Thus to show $\hat{T}$ is semistandard we need to show that entries weakly decrease along rows (Q1), anchor entries form a semistandard composition tableaux (Q2), and free entries are in the possible highest row (Q3).

{\sf (Q1) is still valid: }
Since the rows of $T$ are weakly decreasing and $T$ has no $i+1$s, replacing the $i$s with $i+1$s does not break weakly decreasing.
For the same reason, the anchor entries of the first column of $\hat{T}$ are still strictly increasing and entries still do not repeat in a column.
Thus, it only remains to show that inversion triples are still inversion triples and free entries are still in the highest possible row.

{\sf (Q2) is still valid: }
We show replacing an $i$ with $i+1$ does not turn an inversion triple into a coinversion triple. 
Suppose we have $b < a \leq c$ with $b = i$.  Since $T$ has no $i+1$, $a \geq i+2$ and when we replace $b$ with $i+1$, we still have $b < a \leq c$.
Furthermore suppose we have $a \leq c < b$ with $b = i$.
Then clearly we still have $a \leq c < b$ when $b$ is replaced by $i+1$.
Similar arguments work when $a$ or $c$ (or both) is $i$.

{\sf (Q3) is still valid: }
Finally suppose $i \in {\tt free}(T(r,c))$ and we want to show that the $i+1 \in {\tt free}(\hat{T}(r,c))$ is in the highest possible row.
Suppose to the contrary there exists $r' < r$ such that $i+1$ could be free in box $(r',c)$ of $\hat{T}$.  
Then \[ \max(\hat{T}(r',c)) > i+1 \geq \max(\hat{T}(r',c+1)). \] 
However, since $\max(\hat{T}(r',c)) > i+1$, $\max(\hat{T}(r',c)) = \max(T(r',c))$ and thus \[\max(T(r',c)) > i. \]
Furthermore, either $\max(\hat{T}(r',c+1)) = i+1$ and so $\max(T(r',c+1)) = i$, or since $\hat{T}$ has no $i$s, $i > \max(\hat{T}(r',c+1)) = \max(T(r',c+1))$.
In either case, we have \[\max(T(r',c)) > i \geq \max(T(r',c+1)), \] violating that $T$ was a set-valued semistandard composition tableaux as $i$ should have been free in $(r',c)$.

A very similar argument can be applied to a tableaux with no $i$s and replacing all $i+1$s with $i$s, and thus $\mathcal{G}_\alpha$ is in fact quasisymmetric.
\end{proof}

We now complete the proof of Theorem \ref{thm:QuasiGrothBasis} and Proposition \ref{prop:Refinement}.  Let $QSym_n$ be the ring of quasisymmetric polynomials in $n$ variables.

\begin{proposition}
For compositions $\alpha$ with at most $n$ parts, $\{ \mathcal{G}_\alpha(x_1,...,x_n,0,...) \}$ forms a basis of $QSym_n$.
\label{prop:quasiBasisN}
\end{proposition}
\begin{proof}
Given a quasisymmetric polynomial $f$ in $n$ variables, let $\gamma$ be the lexicographically largest weak composition such that $x^\gamma$ appears in $f$ with nonzero coefficient.
For a composition $\alpha$ with at most $n$ parts, let $\alpha'$ be the weak composition formed by adding $0$s to the beginning of $\alpha$ until $\alpha'$ has $n$ parts.
If $\gamma^+ = \alpha$, we claim $\gamma = \alpha'$.

Since $f$ is quasisymmetric, \[ [x^\gamma] f = [x^\delta] f \] for all $\delta$ with at most $n$ parts such that $\gamma^+ = \delta^+$.  
However, $\alpha'$ is the lexicographically largest such $\delta$ such that $\gamma^+ = \delta^+$, and since $\gamma$ was the lexicographically largest term appearing in $f$, we have $\gamma = \alpha'$.

By the definition of $\mathcal{G}_\alpha$, 
\[ \mathcal{G}_\alpha(x_1,...,x_n,0,...,) = \sum_{\gamma^+ = \alpha} \mathcal{L}_\gamma \] where the sum runs over $\gamma$ with at most $n$ parts.
By lemma \ref{lem:atomLexLargest}, the lexicographically largest monomial of $\mathcal{G}_\alpha$ is $x^{\alpha'}$, and since every $\gamma$ is the sum above has the same maximum part length, the power of each variable is bounded for every term of $\mathcal{G}_\alpha$.

Thus by the same argument as the proof of theorem \ref{prop:lasBasis}, for each $\alpha'$, there are finitely many possibilities of terms that can be introduced by subtracting $[x^{\alpha'}] f \cdot \mathcal{G}_\alpha(x_1,...,x_n,0,...)$ as only lexicographically smaller terms with bounded powers of each variable are introduced. 
Then the process of subtracting $\mathcal{G}_\alpha$ with appropriate coefficient for $x^{\alpha'}$ the lexicographically largest term appearing will terminate and any quasisymmetric polynomial in $n$ variables can be expanded in the $\mathcal{G}_\alpha$s.  

Finally, the monomial quasisymmetric basis for $QSym_n$ is indexed by compositions $\alpha$ with at most $n$ parts and since we have a spanning set with the same (finite) cardinality, $\{ \mathcal{G}_\alpha \}$ as $\alpha$ runs over compositions with at most $n$ parts is a basis for $QSym_n$.
\end{proof}

\QuasiGroth*

\begin{proof}
Let $f$ be a quasisymmetric function.  By proposition \ref{prop:quasiBasisN}, for any $n$,
\[ f(x_1,...,x_n,0,...) = \sum_\alpha c_\alpha\mathcal{G}_\alpha(x_1,...,x_n,0,...). \]
The expansion of the piece of $f$ of degree at most $n$ is determined by $f(x_1,...,x_n,0,...)$.  
Thus as $n \rightarrow \infty$, the expansion of $f$ into $\mathcal{G}_\alpha$ stabilizes and $\{ \mathcal{G}_\alpha \}$ is a basis for $QSym$.
\end{proof}

These new bases provide a method for determining when a function is quasisymmetric (resp. symmetric), and then furthermore $\mathcal{G}_\alpha$-positive (resp. $G_\lambda$-positive) . 

\Refinement*
\begin{proof}
Consider  \[ f  = \sum_\alpha c_\alpha \mathcal{G}_\alpha = \sum_\alpha c_\alpha \sum_{\gamma^+ = \alpha} \mathcal{L}_\gamma. \]
Since $\{ \mathcal{G}_\alpha \}$ is a basis of $QSym$, $f$ is quasisymmetric if and only if it has an expansion in the $\mathcal{G}_\alpha$'s, and as above if and only if $c_\gamma = c_\delta$ for all $\gamma^+ = \delta^+$.
Furthermore, in this case, $f$ is $\mathcal{G}_\alpha$-positive if and only if it is $\mathcal{L}_\gamma$-positive.
Likewise, consider 
\[ f = \sum_\lambda c_\lambda G_\lambda = \sum_\lambda c_\lambda \sum_{\lambda(\gamma) = \lambda}  \mathcal{L}_\gamma.\]
By the same argument, $f$ is symmetric if and only if $c_\gamma = c_\delta$ for all $\lambda(\gamma) = \lambda(\delta)$ and if $f$ is symmetric, $f$ if $G_\lambda$-positive if and only if it is $\mathcal{L}_\gamma$-positive.
\end{proof}

\section{Schur Expansion of $G_\lambda$}
We now provide a new link between ordinary and set-valued tableaux by giving a bijection between the combinatorial objects in C. Lenart's Schur expansion of $G_\lambda$ and semistandard set-valued tableaux.  This section is independent from the remainder of the paper.
Let \[ {\sf S}(\lambda) = \{ F : F \text{ is a semistandard set-valued tableaux of shape } \lambda \}\] and \[ {\sf L}(\lambda) = \left\{ (T,U) : 
\begin{array}{l} T \text{ is row and column strict of shape $\mu / \lambda$ } \\ \text{ \hspace*{.25in} with entries of row $i$ between 1 and $i-1$, and} \\ U \text{ is semistandard of shape $\mu$} \end{array} \right\}. \] 
Then we have the following two decompositons of $G_\lambda$ due to C. Lenart and A. Buch, respectively.
\begin{theorem}[\cite{Lenart2000}, Theorem 2.2]
\[ G_\lambda = \sum_{\lambda \subseteq \mu} \beta^{|\mu| - |\lambda|} g_{\lambda,\mu}s_\mu  = \sum_{(T,U) \in {\sf L}(\lambda)} \beta^{|T|} x^U \] where $g_{\lambda,\mu} = \#\{ T : (T,U_0) \in L(\lambda) \}$ for $U_0$, a fixed semistandard tableaux of shape $\mu$.
\end{theorem}
\begin{theorem}[\cite{Buch2002}, Theorem 3.1]
\[ G_\lambda = \sum_{F \in {\sf S}(\lambda)} \beta^{|F|-|\lambda|}x^F. \]
\end{theorem}

We give a bijection $ {\tt uncrowd}: {\sf L}(\lambda) \rightarrow {\sf S}(\lambda)$ using a repeated application of the recent uncrowding operation of V. Reiner, B. Tenner, and A. Yong \cite{ReinTenYong16}.
Given a set-valued tableaux $F$ of shape $\lambda$, begin with $T = \lambda/\lambda$ and $U = F$.  
Read the boxes of $U$ from bottom to top, right to left.  
While the current box has more than one number, \emph{uncrowd} the box by iteratively removing the largest number from the box and RSK-inserting into the row below.  
During each step, a box will be added to $U$, and in the corresponding box of $T$ record $k - i$ where $k$ is the row of the new box and $i$ is the original row of the number inserted.

\begin{example}
Let $F = \begin{ytableau} 1 & 124 & 4 \\ 45 \end{ytableau}$.  Then ${\tt uncrowd}(F)$ is calculated as follows:
\[\left(\begin{ytableau} *(gray) \, & *(gray)\, & *(gray)\, \\  *(gray)  \end{ytableau}\,, \begin{ytableau} 1 & 124 & 4 \\ 45 \end{ytableau}\right) \Rightarrow
\left(\raisebox{.15in}{\begin{ytableau} *(gray) \, & *(gray)\, & *(gray)\, \\  *(gray)  \\ 1\end{ytableau}}\,, \raisebox{.15in}{\begin{ytableau} 1 & 124 & 4 \\ 4 \\ 5 \end{ytableau}}\right) \Rightarrow\]
\[
\left(\raisebox{.15in}{\begin{ytableau} *(gray) \, & *(gray)\, & *(gray)\, \\  *(gray) & 1  \\ 1\end{ytableau}}\,, \raisebox{.15in}{\begin{ytableau} 1 & 12 & 4 \\ 4  & 4\\ 5 \end{ytableau}}\right) \Rightarrow
\left(\raisebox{.275in}{\begin{ytableau} *(gray) \, & *(gray)\, & *(gray)\, \\  *(gray) & 1  \\ 1 \\ 3\end{ytableau}}\,, \raisebox{.275in}{\begin{ytableau} 1 & 1 & 4 \\ 2  & 4\\ 4 \\ 5 \end{ytableau}}\right)
= {\tt uncrowd}(F)
\]
\end{example}

\begin{theorem}
The map ${\tt uncrowd}$ is a bijection from ${\sf S}(\lambda)$ to ${\sf L}(\lambda)$.
\end{theorem}
\begin{proof}
We prove this by showing that ${\tt uncrowd}$ is a well-defined map to ${\sf L}(\lambda)$ and by constructing its inverse, ${\tt crowd}$.

\noindent{\sf ${\tt uncrowd}$ is well-defined:} 
The algorithm is defined since by uncrowding $U$ from bottom to top, we always insert values into rows that are no longer set-valued and thus RSK is defined.
Furthermore, when the algorithm terminates, $U$ is no longer set-valued as each box is uncrowded until it has exactly one entry.  
Thus we need to show that after each step of ${\tt uncrowd}$, $U$ is semistandard and $T$ is row and column strict with the entries of row $i$ between 1 and $i-1$.
Note that initially these properties hold since $U=F$ is semistandard and $T$ is empty.
Then let $U_i$ be the tableaux formed by the rows of $U$ strictly below row $i$.
\begin{itemize}
\item $U$ has partition shape $\mu$: 
Consider uncrowding an entry $x$ from row $i$.
Since $U$ is column strict, either there is no box below the entry $x$ or the box below the entry $x$ contains $y > x$.
Thus RSK-insertion will either add a box to row $i+1$ weakly left of the box of $x$, or will bump from a box weakly left of the box of $x$.
Thus RSK-insertion cannot add a box to row $i+1$ if it has the same length as row $i$.  
If it does not add a box to row $i+1$, the remaining operation is the classical RSK operation on $U_i$, and thus again $U$ still has partition shape.
 \\

\item $U$ is semistandard: 
Since RSK produces semistandard tableaux, after uncrowding $x$ from row $i$, $U$ is set-valued semistandard up to and including row $i$ and $U_i$ is semistandard.
Thus the only possible concern is that column-strictness is broken between rows $i$ and $i+1$.
Since $x$ is the largest extra entry in its box, it is weakly bigger than all values weakly to the left in row $i$.  
Then, if $x$ ends up in row $i+1$ weakly left of its original position, column-strictness is preserved.
However as above, before uncrowding $x$, the position below $x$ either must be empty or contain $y > x$ and thus column-strictness is preserved.
\\

\item $T$ has shape $\mu / \lambda$: By definition, the skew-shape of $T$ is $\lambda$.  Then, boxes to $T$ are added exactly where they are added to $U$, and so they have the same outer-shape.
\\

\item Row $k$ of $T$ has entries between 1 and $k-1$: The entries of row $k$ of $T$ are $k-i$ where $i$ is the row a value that ended in row $k$ originated in.  
Since $1 \leq i \leq k-1$ for values that end in row $k$, the values of row $k$ of $T$ are between $1$ and $k-1$. 
\\

\item $T$ is row strict: 
As row $i$ of $U$ is uncrowded from right to left, a series of strictly decreasing entries are inserted into $U_i$. 
Thus by Fulton's Row Bumping Lemma (pg. 9 \cite{Fulton1997}), a vertical strip is added to $U_i$ in this process.
Thus for each row $k$ of $T$, $k-i$ appears at most once since at most once value from row $i$ can end in row $k$.  
Furthermore, the entries already in row $k$ of $U$ must have come from rows $i+1,...,k-1$ and thus before uncrowding row $i$, the entries of row $k$ of $T$ can only be $1, ..., k-(i+1)$. 
Since the new box is placed at the end of row $k$ of $T$ and given the value $k-i$, $T$ remains row strict.
 \\ 

\item $T$ is column strict: Again consider uncrowding $x$ from row $i$ where $x$ ends up in a new box row $k$.  
The box in row $k-1$ above the new box must have been added before this point and so initially must have come from rows $i, ..., k-2$.
Thus the value in this box in row $k-1$ must be one of $k-1-(k-2), ..., k-1-i$ and recording $k-i$ in the new box preserves column strictness.
\end{itemize}

\noindent {\sf Definition of ${\tt crowd}$:} We now define the crowding operation ${\tt crowd}: {\sf L}(\lambda) \rightarrow {\sf S}(\lambda)$ again by iterating the process of V. Reiner, B. Tenner, and A. Yong.
Given a pair $(T,U)$, consider the tableaux $\tilde{T}$ where each $x$ in row $k$ of $T$ is replaced with $k - x$.  
Thus $\tilde{T}$ has strictly decreasing rows and weakly decreasing columns, and so the minimal value of $\tilde{T}$ appears at an inner corner (and possibly elsewhere).
Let $i$ be the minimal value of $\tilde{T}$ and choose the lowest inner corner of $\tilde{T}$ containing $i$.  
Reverse-RSK from this corner in $U$, but stop after reverse-bumping a value, say $x$, out of row $i+1$.  
Instead of bumping a value out of row $i$, add $x$ to the unique box $b$ of row $i$ such that ${\tt max}(b) < x \leq {\tt min}(\text{box right of } b)$.
Then remove the now empty corner from $U$ and the corresponding box from $\tilde{T}$.
Repeat this process until $\tilde{T}$ is empty.

\noindent {\sf ${\tt crowd}$ is well-defined}: 
This algorithm terminates when $\tilde{T}$ has no further inner corners, i.e. when $U$ is a set-valued tableaux of shape $\lambda$. 
Furthermore, since rows are crowded from top to bottom, the reverse-RSK part of the algorithm only operates on rows that are not set-valued, and thus is well-defined.  
Thus the only thing that needs to be shown is that when adding a value $x$ to row $i$ (instead of bumping from row $i$) row $i$ of $U$ remains weakly increasing and the columns of $U$ remain strictly increasing between rows $i$ and $i+1$.

Consider crowding all boxes corresponding to $i$ in $\tilde{T}$.  
Since $\tilde{T}$ has strictly decreasing rows, these boxes must form a vertical strip that is removed from $U_i$ from bottom-to-top by reverse-RSK.
Thus since RSK and reverse-RSK are inverses, by the Row Bumping Lemma, entries are added to row $i$ in strictly increasing order.  
From the definition of ${\tt crowd}$, row $i$ remains weakly increasing if such a box $b$ exists, so we just need to show that $b$ exists.  
However before being bumped, $x$ was in row $i+1$ and thus strictly bigger than ${\tt max}(b_0)$ where $b_0$ is the box directly above $x$.  
By above, $x$ is bigger than any extra entry of row $i$ and so if $x > {\tt min}(b')$ for some box $b'$ of row $i$, $x > {\tt max}(b')$.
Thus $b$ is the rightmost box of row $i$ such that $x > {\tt min}(b)$, which exists as $b_0$ is one such box.

Now we simply need to show that $x$ is smaller than the value in the box immediately below $x$.
By above, $x$ ends up weakly right of where it was originally.  
Furthermore, if $x$ got bumped out by $y$, then $y > x$ and so $x$ is less than everything weakly right of its original position.
Thus column strictness is preserved.

\noindent{\sf {\tt uncrowd} and ${\tt crowd}$ are weight-preserving:} This is clear as numbers in $U$ are moved but not changed.

\noindent{\sf {\tt uncrowd} and ${\tt crowd}$ are mutual inverses:} When uncrowding row $i$, boxes are added to $U$ from top to bottom.  
Thus to reverse this process, when crowding into row $i$, boxes are removed from $U$ from bottom to top.  
Similarly, rows are uncrowded from bottom to top and thus to reverse this process, rows are crowded from top to bottom.  
Finally when uncrowding, the maximum entry of a box is removed and thus when crowding the new entry is always the maximum in its box. 

\end{proof}

\section{Conjectures}
We have defined the Lascoux atoms combinatorially in terms of set-valued skyline fillings, but there is also a natural definition based on isobaric divided difference operators.
Let $s_i$ act on polynomials by switching $x_i$ and $x_{i+1}$.
Then, we have the operators 
\[ \partial_i = \frac{1-s_i}{x_i - x_{i+1}} \hspace*{1in} \pi_i = \partial_i x_i \hspace*{1in} \hat{\pi}_i = \pi_i - 1. \]
These operators all satisfy the braid relations and so given $w \in S_n$, we define $\partial_w$ by $\partial_w = \partial_{a_1}...\partial_{a_k}$ (and $\pi_w$ and $\hat{\pi}_w$ analogously) where $a_1...a_k$ is any reduced word of $w$.
Given a composition $\gamma$, let $w(\gamma)$ be the shortest permutation that sends $\lambda(\gamma)$ to $\gamma$.  For example, for $\gamma = 1021$, $\lambda(\gamma)  = 211$ and $w(\gamma) = 3142$.  

The Demazure character is $\kappa_\gamma = \pi_{w(\gamma)} x^{\lambda(\gamma)}$ and the Demazure atom is $A_\gamma = \hat{\pi}_{w(\gamma)} x^{\lambda(\gamma)}$ where $x^\lambda = x_1^{\lambda_1} x_2^{\lambda_2}...$.
In \cite{Lascoux2001}, Lascoux defined $K$-theoretic deformations of the Demazure characters using modified operators that still satisfy the braid relations: 
\[ \tilde{\partial}_i = \partial_i(1 + \beta x_{i+1}) \hspace*{1in} \tau_i = \pi_i(1+\beta x_{i+1}) \hspace*{1in} \hat{\tau}_i = \tau_i -1. \]

Then the Lascoux polynomial is $\Omega_\gamma = \tau_{w(\gamma)} x^{\lambda(\gamma)}$ and the Lascoux atom is $\hat{\mathcal{L}}_\gamma = \hat{\tau}_{w(\gamma)} x^{\lambda(\gamma)}$.
By manipulating the operators above, we obtain the following decomposition of the Lascoux polynomial into Lascoux atoms that matches the Demazure case.

\begin{theorem}
 \[ \Omega_\delta = \sum_{\gamma \leq \delta} \hat{\mathcal{L}}_\gamma \] where $\gamma \leq \delta$ if $\lambda(\gamma) = \lambda(\delta)$ and $w(\gamma) \leq w(\delta)$ in Bruhat order.
\end{theorem}
\begin{proof}
In sections 2 and 3 of \cite{Pun16}, A. Pun gives a proof in the case that $\beta = 0$ using relations derived between $\partial_i, \pi_i$, and $\hat{\pi}_i$. 
To extend this proof line by line to the Lascoux case, we only need to show $\hat{\tau}_i\hat{\tau}_i = -\hat{\tau}_i$.

We first show that $\tau_i\tau_i = \tau_i$.
To do this, consider $\pi_i x_{i+1}f = \partial_i (x_i x_{i+1} f)$.  Since $x_ix_{i+1}$ is symmetric in $i$ and $i+1$, 
\[ \pi_i x_{i+1}f = \partial_i (x_i x_{i+1} f) = x_ix_{i+1} \partial_i f. \]
Now, from Proposition 3.1 of \cite{Pun16}, $\pi_i^2 = \pi$ and $\partial_i \pi_i = 0$.
Then
\begin{align*}
\tau_i^2  &= (\pi_i(1 + \beta x_{i+1}))(\pi_i(1 + \beta x_{i+1})) \\ 
&= \pi_i^2(1 + \beta x_{i+1}) + \pi_i \beta x_{i+1} \pi_i ( 1 + \beta x_{i+1}) \\
&= \pi_i(1 + \beta x_{i+1}) + \beta x_i x_{i+1} \partial_i \pi_i ( 1 + \beta x_{i+1}) \\
&= \tau_i + 0.
\end{align*}

Since $\tau_i = 1 + \hat{\tau}_i$ and $\tau_i^2 = \tau_i$,  \[ 1 + \hat{\tau}_i = (1 + \hat{\tau}_i)^2 = 1 + 2\hat{\tau}_i + \hat{\tau}_i^2. \]
Thus \[ \hat{\tau}_i^2 = -\hat{\tau}_i. \]
\end{proof}

A conjectural combinatorial model for $\Omega_\gamma$ using $K$-Kohnert diagrams was given by C. Ross and A. Yong in \cite{RossYong2015}, but there are no proven combinatorial rules for $\Omega_\gamma$ or $\hat{\mathcal{L}}_\gamma$.
However, we have checked the following conjectures for all weak compositions $\gamma$ with at most 8 boxes and at most 8 rows, both which generalize the Demazure case.
\begin{conjecture}
\[ \hat{\mathcal{L}_\gamma} = \mathcal{L}_\gamma = \sum_{F \in {\sf SetSkyFill}(\gamma)} \beta^{|F|-|\gamma|} x^F. \] 
\label{conj:LasConj}
\end{conjecture}

\begin{conjecture}
\[ \Omega_\gamma = \sum_F \beta^{|F| - |\gamma|} x^F \] where the sum runs over all semistandard set-valued skyline fillings of shape $\gamma^*$  and basement $b_i = n-i+1$.
\end{conjecture}

In \cite{HLMvW2011}, J. Haglund, K. Luoto, S. Mason, and S. van Willigenburg refine the Littlewood-Richardson rule to give the expansion of $A_\gamma \cdot s_\lambda$ into Demazure atoms.  O. Pechenik and A. Yong \cite{PechYong16} develop the theory of genomic tableaux to describe multiplication in $K$-theory.  We conjecture the natural genomic analogue of the rule of J. Haglund et. al. extends to Lascoux atoms.

When $\delta, \gamma$ are weak compositions with $\gamma_i \leq \delta_i$ for all $i$, a \emph{skew skyline diagram} of shape $\delta/\gamma$ is formed by taking the skyline diagram of shape $\delta$ and given basement and extending the basement into the cells of $\gamma$. If $n$ is the largest entry allowed in the filling, a \emph{large basement} is such all basement entries of the basement are larger than $n$ and decrease from top to bottom.  
As seen in \cite{HLMvW2011}, with a large basement, the exact basement entries do not determine valid skyline fillings  and thus we denote it by $`*'$.

A \emph{genomic} filling is a filling of $\delta/\gamma$ with labels $i_j$ where $i$ is a positive integer and for each $i$, $\{ j | i_j \text{ appears in the filling} \} = \{ 1, ..., k_i \}$ for some nonnegative integer $k_i$.
The set of labels $\{ i_j \}$ for all $j$ is the \emph{family} $i$, while the set of all labels $i_j$ for fixed $i$ and $j$ is the \emph{gene} $i_j$.
The \emph{content} of a genomic filling is $(k_1,k_2,...)$.
The \emph{column reading word} of a skyline filling reads the entries of the boxes (excluding the basement) in columns from top to bottom, right to left.
A genomic filling is \emph{semistandard} if 
\begin{itemize}
\item[(G1)] at most one entry from a family (resp. gene) appears in a column (resp. row) 
\item[(G2)] the label families are weakly decreasing along rows,
\item[(G3)] every triple with three distinct genes is an inversion triple comparing families, and
\item[(G4)] for every $i$, the genes appear in weakly decreasing order along the reading word. 
\end{itemize}

A word is \emph{reverse lattice} if at any point and any $i$ we have always read more $i+1$s than $i$s.
A genomic filling is \emph{reverse lattice} if for any selection of exactly one label per gene, the column reading word is reverse lattice.

\begin{conjecture}
\[ \mathcal{L}_\gamma \cdot G_\lambda = \sum_\delta \tilde{a}_{\gamma,\lambda}^\delta \mathcal{L}_\delta \]
where $\tilde{a}_{\gamma,\lambda}^\delta$ is the number of reverse lattice, genomic semistandard skyline fillings of skew-shape $\delta/\gamma$ (using a large basement) with content $\lambda^*$.
\end{conjecture}

\begin{example}
$\tilde{a}_{102,21}^{314} = 2$ and the two witnessing fillings are
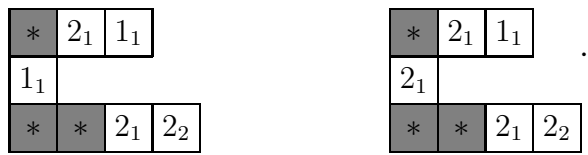
\begin{center}
$\begin{ytableau}
 *(gray) *  & 2_1 & 1_1 \\
 1_1 \\
 *(gray) *  & *(gray) *  & 2_1 & 2_2 \\
\end{ytableau}
\hspace*{1in}
\begin{ytableau}
 *(gray) *  & 2_1 & 1_1 \\
 2_1 \\
 *(gray) *  & *(gray) *  & 2_1 & 2_2 \\
\end{ytableau}$.
\end{center}
%The following tableaux is a genomic semistandard skyline filling of appropriate shape and content but is not lattice:
%$\begin{ytableau}
% *(gray) *  & 2_2 & 1_1 \\
% 2_1 \\
% *(gray) *  & *(gray) *  & 2_2 & 1_1 \\
%\end{ytableau}$.
\end{example}

\section*{Acknowledgements}
I thank my advisor Alexander Yong for his initial question that inspired this project, his support and encouragment throughout, and his advice that greatly improved the exposition.  I also thank Steph van Willigenburg and Sarah Mason for helpful conversations.  I used SAGE extensively throughout my investigation and was supported by a NSF grant during the initial work.

\bibliography{AtomReferences}
\bibliographystyle{mybst}

\end{document}